\documentclass[10pt]{amsart}

\usepackage{amsmath}
\usepackage{amssymb}
\usepackage{amsthm}
\usepackage{amsfonts}
\usepackage{array}
\usepackage{euscript}
\usepackage{graphics,graphicx}
\usepackage{cancel}
\usepackage{fancybox}
\usepackage{verbatim}
\usepackage{paralist}
\usepackage{setspace}
\usepackage{fancyhdr}
\usepackage{hyperref}
\usepackage{url}
\usepackage[margin=2cm]{geometry}
\usepackage[normalem]{ulem}

\usepackage{tikz}
\usetikzlibrary{arrows,matrix, patterns}

\newtheoremstyle%
{Theorem}%
{}%
{}%
{\itshape}%
{}%
{}%
{.}%
{ }%
{\thmname{\bfseries #1}%
\thmnumber{\;\bfseries #2}%
\thmnote{\;(\bfseries #3)}}%

\theoremstyle{Theorem}
\newtheorem{thm}{Theorem}[section]

\newtheorem{lem}[thm]{Lemma}
\theoremstyle{definition}
\newtheorem{dff}[thm]{Definition}
\newtheorem{xmp}[thm]{Example}

\newtheorem{rmk}[thm]{Remark}

\newcommand{\FF}{\mathbf F}
\newcommand{\ZZ}{\mathbf Z}
\newcommand{\QQ}{\mathbf Q}
\newcommand{\NN}{\mathbf N}

\newcommand{\id}{\operatorname{id}}

\newcommand{\fm}{\mathfrak{m}}
\newcommand{\cA}{\mathcal{A}}
\newcommand{\cE}{\mathcal{E}}
\newcommand{\cF}{\mathcal{F}}

\newcommand{\cP}{\mathcal{P}}
\renewcommand{\d}{\delta}
\def\ur{\textrm{ur}}
\def\alg{\textrm{alg}}
\def\ux{\underline{x}}
\def\cs{s}

\def\cN{\mathcal{N}}

\def\Gal{\operatorname{Gal}}

\colorlet{DG}{green!50!black}
\colorlet{DO}{orange!50!black}
\colorlet{DR}{red!50!black}
\colorlet{DB}{blue!50!black}
\colorlet{DP}{purple!50!black}

\title{Arithmetic partial differential operators on ramified extensions of $\ZZ_p$}
\author{Marshall Donn}
\address{Mathematical Sciences Bldg, MATH 605,
150 N. University St.,
West Lafayette, IN 47097}
\email{mdonn@purdue.edu} 

\author{Lance Edward Miller}
\address{Department of Mathematical Sciences,  SCEN 355,
University of Arkansas, 
Fayetteville, AR 72701}
\email{lem016@uark.edu}

\author{William D.\ Taylor}
\address{Department of Mathematical Sciences, Boswell 316B,
Tennessee State University, 
Nashville, TN 37209}
\email{wtaylo17@tnstate.edu}

\begin{document}

\begin{abstract}
The notion of $p$-derivation as introduced by Buium has a rich history of applications in arithmetic geometry. Working over $\ZZ_p$, Buium-Ralph-Simanca showed that arithmetic differential operators built from these determine $p$-adic analytic functions and vice versa. Notably, over $\ZZ_p$, there is only one $p$-derivation. In this article, we consider the same question for ramified extensions $A_\pi$ with uniformizer $\pi$. This has the effect of passing from ordinary to partial differential operators, since such extensions can enjoy multiple $\pi$-derivations. We introduce a notion of analytic functions which are naturally determined by these operators in a way analogous to that in the unramified case, but with a significantly richer structure.  We show that under mild conditions, analytic functions and partial differential operators determine each other in this setting.

\end{abstract}
\maketitle

\section{Introduction}

We build on an arithmetic analog of differential equations pioneered by Buium and introduced in \cite{Bui96}, for a thorough treatment see \cite{Bui05} and the subsequent references. This subject rests on the development of a remarkable analogy between differentiation and the Frobenius map. Fix throughout a prime integer $p$. Recall a {\bf lift of Frobenius} on a ring $A$ is a ring homomorphism $\phi \colon A \to A$ so that $\phi(x) \equiv x^p \bmod p$. When $A$ is $p$-torsion free, associated to $\phi$ is a {\bf $p$-derivation} $\delta \colon A \to A$ defined by 
$$\delta (a) := \frac{\phi(a) - a^p}{p}\, \textrm{ for } a \in A.$$ 
This is interpreted as an arithmetic differentiation or as in line with the perspectives here, a differentiation in the `arithmetic direction' $p$. A basic example to consider is $A = \ZZ_p$ and $\phi = \id$. Passing to the completed maximal unramified extension of $\QQ_p$ it is important to note that its ring of integers, denoted $R$, still supports a unique such ring homomorphism. That is, there is only {\it one} $p$-derivation on $R$. Viewing polynomial expressions in $\d$ as analogs of differential equations, these are {\it ordinary differential equations}. This theory has generated a substantial number of applications to arithmetic geometry via finiteness type theorems in Diophantine geometry \cite{Bui96,BP09}. These applications come largely from the arithmetic analogs of Manin maps \cite{Bui95} as well as the theory of $\delta$-modular forms \cite{Bui00,Bui08}. Through this lens, problems are translated into a theory of ordinary arithmetic differential equations, which then can be solved. Additionally, an entirely new subject aimed at honing this analogy into a lens which transforms tools from differential geometry into tools aimed at purely arithmetic applications was again pioneered by Buium \cite{Bui17,Bui19}. 

\

There is a substantial difference between $\ZZ_p$ and $R$ however. Roughly, \cite{BRS11} shows that, working in $\ZZ_p$ the basic operations coming from $p$-derivations are intimately connected to classic theories of $p$-adic analytic functions. In fact, the main theorem of \cite{BRS11} is that the theories determine each other. This is far from the case when working over $R$, and so the $\d$-operation can be viewed as an algebraic replacement of more analytic techniques. 

\

We give a precise recap of the connections over $\ZZ_p$. In \cite{BRS11}, two kinds of functions are considered. The first are the {\it arithmetic differential operators of order $r$}, those functions $f \colon \ZZ_p \to \ZZ_p$ of the form $f(a) = F(a,\d (a), \ldots, \d^r (a))$ for a restricted power series $F \in \ZZ_p[[x_0,x_1,\ldots,x_r]]$. 
The second are the {\it $p$-adic analytic functions of level $r$}, which are functions $f \colon \ZZ_p \to \ZZ_p$ which for each $a \in \ZZ_p$, there is a restricted power series $F_a \in \ZZ_p[[x]]$ so that $f(a + p^r u) = F_a(u)$. Their main theorem is the following, see \cite[Thm. 4,5]{BRS11}.

\begin{thm}[Buium-Ralph-Simanca] 
A function $f \colon \ZZ_p \to \ZZ_p$ is an arithmetic differential operator of order $r$ if and only if it is analytic of level $r$.  
\end{thm}

 However, this theorem fails if one passes from $\ZZ_p$ to $R$, \cite[Thm. 10.1]{RS12}. The failure highlights that arithmetic differential equations are genuinely new objects from this classic perspective lying outside the realm of analytic methods. Perhaps the best interpretation is that they are sometimes a replacement for the lack of these analytic methods in arithmetic geometry. This in part helps explain their many applications. 

\

Some of the application of arithmetic differential equations to Diophantine geometry were recently enhanced by the second author's works with Buium which largely is unlocked by passing to a {\it ramified} extension. Specifically, the rings $R_\pi := R[\pi]$ for $\pi$ a root of an Eisenstein polynomial over $R$ defining a Galois extension support multiple $\pi$-lifts of Frobenius which should be viewed in analogy to rings of functions of several variables. This admitted immediately a wider space of solutions to arithmetic differential equations. The ramified solutions to arithmetic ordinary differential equations was explored \cite{BM20,BM23a}. It also opened a door to a wider theory of arithmetic differential equations, analogous to moving from ordinary to partial differential equations. Specifically, a study fully embracing the arithmetic {\it partial} differential equations was taken on \cite{BM23b}. Systematic adaptations yielded a significant enhancement to the foundations of arithmetic differential geometry, \cite{Bui17}. None of these are discussed in the content here, but the interested reader should consult \cite{BM26}. Problems in this theory naturally suggested the questions that led to our work. 

\

The relationships outlined in \cite{BRS11} have been adapted to the `multivariate' setting, but not in a way that can be applied to the fully arithmetic PDE situation just described. Indeed, in \cite{Bra20} the setting $\ZZ_p^n$ was considered. However, this passage to a multi-variable situation utilizes a different analogy. The topological ring $\ZZ_p$ can be viewed as both the unit $p$-adic ball as well as the $p$-adic valued functions on a point. One multivariate generalization considered moves from $\ZZ_p$ to $\ZZ_p^n$. This is a `spacial' enhancement, it keeps the level of arithmetic the same but adds more inputs, thus lies in the first analogy. Instead, viewing $p$-derivations as `directions' as they are interpreted in \cite{BM26}, one should consider ramified extensions of $\ZZ_p$. Such a consideration was made \cite{Lam23}, however, this did not consider multiple $\pi$-lifts of Frobenius, instead only using one choice of lift. To fully harness the arithmetic PDE setting we want to explore arithmetic differential operators attached to multiple $\pi$-lifts of Frobenius with minimal restriction. To account for this, we also propose a new and deeper notion of analytic function to which these differential operators correspond.

\

We briefly summarize the main results of this article. To prepare, we fix $\pi$ a root of an Eisenstein polynomial over $\ZZ_p$ defining a Galois extension and set $A_\pi := \ZZ_p[\pi]$. Unlike the unramified setting, $A_\pi$ can enjoy {\it multiple} $\pi$-lifts of Frobenius coming from restrictions of Frobenius automorphisms in the absolute Galois group. These always coincide when restricting further to $\ZZ_p$ but can remain distinct over $A_\pi$. Equip $A_\pi$ with $\pi$-lifts of Frobenius $\phi_1,\phi_2,\ldots,\phi_n$, i.e.\ ring homomorphisms for which $\phi_i(a) \equiv a^p \bmod \pi$. 
Associated to each $\pi$-lift of Frobenius is its $\pi$-derivation $\delta_i(a) := (\phi_i(a) - a^p)/\pi$. To keep track of bookkeeping, set $\mathbb M_n$ the non-commutative monoid of words on the set $\{ 1,\ldots,n\}$ under concatenation and for $r \geq 0$, $\mathbb M_n^r$ the set of words of length at most $r$ and $\mathbb M_n^{=r}$ the words of length exactly $r$. Each word $\mu = i_1 \cdots i_s \in \mathbb M_n$ defines a composition $d_\mu := \d_{i_1} \circ \cdots \circ \d_{i_s}$ and a composition $\phi_\mu = \phi_{i_1} \circ \cdots \circ \phi_{i_s}$. 

\

An {\bf arithmetic partial differential operator} of order at most $r$ is a function $f \colon A_\pi \to A_\pi$ which is given by $f = F(\d_\mu (a) \colon \mu \in \mathbb M_n^r)$ where $F \in A_\pi[[x_\mu \colon \mu \in \mathbb M_n^r]]$ is a restricted power series. A function $f \colon A_\pi \to A_\pi$ is {\bf analytic of level $r$} provided for each $a \in A_\pi$ there is a restricted power series $F_a \in A_\pi[[x_\mu \colon \mu \in \mathbb M_n^r]]$ so that $f(a + \pi^r u) = F_a(\phi_\mu(u))$ for all $u \in A_\pi$. 

\

Note the definitions naturally agree, taking $\pi = p$ and $n = 1$, with those of \cite{BRS11}. Passing from $\ZZ_p$ to $A_\pi$ is an `arithmetic' enhancement, it keeps the level of the space the same but now adds more `arithmetic directions' i.e., more functions which honestly depend on multiple arithmetic directions as indicated by the multiple $\pi$-derivations. This notion of analytic produces some interesting phenomena, namely a function of one variable that can depend multivariately on its input and the non-trivial family of $\pi$-lifts of Frobenius on that variable. These analytic functions are genuinely {\it new} objects to $p$-adic analysis that appear only in the ramified setting, they have no unramified counterpart. It is this analogy that is of most utility in the theories introduced in \cite{BM23b,BM26}. 

\begin{thm}(Theorem~\ref{thm:PADOtoAnal} and Theorem~\ref{thm:analtopado}) Fix $\{ \phi_1,\ldots,\phi_n\}$ a collection of $\pi$-lifts of Frobenius on $A_\pi$ and $r \geq 0$. Every arithmetic partial differential operator of order $r$ is analytic of level $r$.  If $\phi_1$ is identity,
then the converse holds, that is, when $f \colon A_\pi \to A_\pi$ is analytic of level $r$, $f$ is an arithmetic partial differential operator of order at most $r$.
\end{thm}

While the major beats of the proofs are similar to those in \cite{BRS11}, the level of notational precision necessary to pass from the ordinary differential operator to the partial differential operator setting is substantial. Theorem~\ref{thm:PADOtoAnal}, which shows every arithmetic partial differential operator is analytic, is essentially direct to establish. We also establish an order analysis similar to \cite[Lem. 10]{BRS11}, see Theorem~\ref{lma1}. To prove Theorem~\ref{thm:analtopado}, one constructs the desired power series by inductively computing its coefficients. Here, the setting is more complicated than that in \cite{BRS11}, and in fact uses \cite[Thm. 1]{BRS11} roughly as a base case. Additionally, there is an inherent uniqueness in the main theorem of \cite{BRS11} that does {\it not} persist to the multivariate setting. We give, in Section~\ref{sec:analtoapdo}, a self-contained summary of the proof of the corresponding theorem \cite[Thm. 5]{BRS11} which is the $\pi = p, n = 1$ case. We also discuss the new ideas and nuances needed to accomplish our adaptation. 

\

\noindent {\bf Acknowledgments:} We thank Rajat Mishra for preliminary readings and helpful suggestions and A.\ Buium for helpful discussions on the topic. The third author was supported by an AMS-Simons Research Enhancement Grant for Primarily Undergraduate Institution Faculty.

\
	
\section{Preliminaries} We review the setting of purely arithmetic PDEs utilized in \cite{BM23b,BM26}. Fix $\QQ_p^{\alg}$ an algebraic closure of $\QQ_p$. Additionally set $\QQ_p^{\ur}$ the maximal unramified extension of $\QQ_p$ in $\QQ_p^{\alg}$. Denote by $K$ the metric completion of $\QQ_p^{\ur}$ and $R$ for the ring of integers of $K$. Throughout, we denote by $\Pi$ the set of all elements $\pi \in \QQ_p^{\alg}$ which are roots of Eisenstein polynomials with coefficients in $\ZZ_p$ and for which $\QQ_p(\pi)/\QQ_p$ is Galois. 

\

For any $\pi\in \Pi$ write $A_\pi := \ZZ_p[\pi]$ the ring of integers of $\QQ_p(\pi)$. Let $\Gal(K^{\alg}/\QQ_p)$ be the Galois group of $K^{\alg}/\QQ_p$. An element $\phi \in \Gal(K^{\alg}/\QQ_p)$ is a {\bf Frobenius automorphism} provided it is continuous and induces the $p$-power Frobenius on the residue field of $K^{\alg}$. Clearly for $\pi\in \Pi$ the field $K_{\pi}$ is mapped into itself by every Frobenius automorphism $\phi$. By continuity of $\phi$ we have an induced automorphism $\phi \colon A_\pi \to A_\pi$ inducing the $p$-power Frobenius on $A_\pi/\pi A_\pi$, i.e., a {\bf $\pi$-Frobenius lift} on $A_\pi$. 

\

Associated to each $\pi$-Frobenius lift $\phi$ on $A_\pi$ is a {\bf $\pi$-derivation} $\delta \colon A_\pi \to A_\pi$ given by $\d(x) := \pi^{-1}(\phi(x) - x^p)$. Each $\pi$-derivation is a set map such that $\delta(1) = 0$ and  
\begin{eqnarray*}
 \delta(x+y) & = & \delta(x) + \delta(y) + \frac{ x^p + y^p - (x+y)^p}{\pi} \\
 \delta(xy) & = & x^p \delta (y) + y^p \delta (x) + \pi \delta (x) \delta (y).
\end{eqnarray*} 

\noindent Fix $\pi \in \Pi$. A {\bf partial $\delta$-structure} on $A_\pi$ is a choice $(\delta_1,\ldots,\delta_n)$ where each $\delta_i$ is a $\pi$-derivation of $A_\pi$. 

\

Let ${\mathbb M}_n$ be the free non-commutative monoid with identity generated by the set $\{1,\ldots,n\}$, $$\mathbb M_n:=\{0\}\cup \{i_1\ldots i_s\ |\ s\in \mathbb N,\ i_1,\ldots,i_s\in \{1,\ldots,n\}\};$$ its elements will be referred to as {\bf words}. The {\bf length} $|\mu|$ of a word $\mu:=i_1\ldots i_s$ is defined by $|\mu|=s$. Denote by $0$ the {\bf empty word} and set its length $|0|=0$. Multiplication is given by concatenation $(\mu,\nu)\mapsto \mu\cdot \nu$ and $0$ is the identity element. For words $\mu = i_1\ldots i_s$ and $\nu = j_1 \ldots j_t$ we say $\nu$ is a {\bf subword} of $\mu$ provided for each $1 \leq \ell \leq t$ there is $1 \leq k_\ell \leq s$ so that $j_\ell = i_{k_\ell}$ and $k_1< k_2< \cdots < k_t$. We write $\nu \leq \mu$ to denote that $\nu$ is a subword of $\mu$. 

\

By convention, we take $\mathbf{N}$ to be the set of non-negative integers. For all $r\in \mathbf{N}$, let ${\mathbb M}^r_n$ be the set of all elements in ${\mathbb M}_n$ of length $\leq r$. Set $\mathbb M_n^+:=\mathbb M_n\setminus \{0\}$ and $\mathbb M_n^{r,+}:=\mathbb M_n^r\setminus \{0\}$. The quantity $n$ is fixed throughout the paper, so we drop this from the notation when possible, that is we will write $\mathbb M$, $\mathbb M^r$, etc.\ for $\mathbb M_n$, $\mathbb M_n^r$, etc.  

\

 Fix $\phi_1,\ldots,\phi_n$ a family of Frobenius automorphisms which induce $\pi$-Frobenius lifts on $A_\pi$. Set for each $\delta_i$ the corresponding $\pi$-derivation. This gives $A_\pi$ a partial $\delta_\pi$-structure. For each $\mu = i_1 \ldots i_s \in \mathbb M$ we define 
$$\phi_\mu := \phi_{i_1} \circ \cdots \circ \phi_{i_s} \textrm{ and } \delta_\mu := \delta_{i_1} \circ \cdots \circ \delta_{i_s}.$$ 
In the next definition we make reference to {\bf restricted power series} with $A_\pi$-coefficients in some number of variables. This is always taken to mean an element $F$ of $A_\pi[[\ux]]$, where $\ux$  is a family of variables, for which coefficients of $F$ tend to zero $\pi$-adically. Specifically, this will occur often in variables indexed by words. To this end, if $x$ is a variable, set $\{ x_\mu \colon \mu \in \mathbb M\}$ a set of new variables.  For a set of words $E\subset \mathbb M ^r$, let $\Gamma_{E} \subset A_\pi[[x_\mu\colon \mu \in E]]$ be the set of restricted power series over $A_\pi$ in the variables $x_\mu$ for $\mu\in E$.  Let $\Delta_E:\Gamma_{E}\to \mathrm{Fun}(A_\pi,A_\pi)$ be the ring homomorphism taking $x_\mu$ to $v\mapsto \delta_\mu(v)$ and extending $A_\pi$-linearly.  Similarly, let $\Phi_E:\Gamma_{E}\to \mathrm{Fun}(A_\pi,A_\pi)$ be the ring homomorphism taking $x_\mu$ to $v\mapsto \phi_\mu(v)$ and extending $A_\pi$-linearly.  Thus, to be precise, we have for any $G = \sum_{j\in \NN^{E}} c_j\prod_{\mu\in E}x_\mu^{j(\mu)}\in \Gamma_{E}$,
\[
    \Delta_E(G)(v)  =  \sum_{j\in \NN^{E}} c_j\prod_{\mu\in E}\delta_\mu(v)^{j(\mu)} \quad \text{and}\quad    \Phi_E(G)(v) =  \sum_{j\in \NN^{E}} c_j\prod_{\mu\in E}\phi_\mu(v)^{j(\mu)}
\]
In the case $E = \mathbb M^r$, which is our most common case, we suppress the subscript.  That is, $\Gamma:=\Gamma_{\mathbb M^r}$, $\Phi:=\Gamma_{\mathbb M^r}$, and $\Delta :=\Delta_{\mathbb M^r}$.

\begin{dff}\label{dff1} Fix a partial $\delta_\pi$-structure on $A_\pi$ and $r \geq 0$ coming from $\{\phi_1,\ldots,\phi_n\}$ a set of Frobenius automorphisms. A function $f \colon A_\pi \to A_\pi$ is an {\bf arithmetic partial differential operator of order at most $r$} provided there exists $G\in \Gamma$ such that $f = \Delta(G)$. 

\

\noindent Fix $a \in A_\pi$ and $r \geq 0$. Call a function $f \colon A_\pi \to A_\pi$ {\bf analytic at $a$ of level $r$} provided there exists $F_a \in \Gamma$ such that $f(a + \pi^r u) = \Phi(F_a)(u)$ for all $u\in A_\pi$. We say {\bf $f$ is analytic of level $r$} if $f$ is analytic of level $r$ at all $a \in A_\pi$.
\end{dff}

We can describe the analytic functions using the notation $\Phi$ above.  First we note that if $f\colon A_\pi\to A_\pi$ is analytic of level $r$ at some $a\in A_\pi$, then it is analytic of level $r$ at any $a'$ in the ball $a + \pi^rA_\pi$.  Indeed, for any such $a'$, $a' = a+ \pi^r u'$ for some $a\in \cA$.  Furthermore, since $f$ is analytic of level $r$ at $a$, there exists $F_a\in \Gamma$ such that for all $u\in A_\pi$, $f(a+\pi^r u) = \Phi(F_a)(u)$.  Therefore, for any $u\in A_\pi$, $f(a'+\pi^r u) = f(a+\pi^r(u + u')) = \Phi(F_a)(u + u') = \Phi(F_{a'})(u)$, where $F_{a'}$ is the transformation of $F_a$ under the rule $x_\mu \mapsto x_\mu + \phi_\mu(u')$.  Hence $f$ is analytic of level $r$ at $a'$.

\ 

Therefore, to show that a function $f:A_\pi\to A_\pi$ is analytic of level $r$, it suffices to show that $f$ is analytic of level $r$ at a set of centers of balls of radius $1/\pi^r$ which cover $A_\pi$.  Let $\cA\subset A_\pi$ be a such a collection of centers, that is, any $v\in A_\pi$ can be written uniquely as $v = a+\pi^r u$ for some $a\in \cA$, $u\in A_\pi$.  By the notation $\Phi_E^{\oplus \cA}$ we mean the function
$\Phi_E^{\oplus \cA} \colon \Gamma_E^{\oplus \cA}\to \mathrm{Fun}(A_\pi,A_\pi)$
given by, for all $u\in A_\pi$,
\begin{align*} 
\Phi_E^{\oplus \cA} \colon \Gamma_E^{\oplus \cA} \to & \mathrm{Fun}(A_\pi,A_\pi) \\ 
(F_a)_{a\in \cA}\mapsto & \left(a + \pi^r u \mapsto \Phi_E(F_a)(u)\right).
\end{align*}
Thus, the set of analytic functions of level $r$ is precisely the image of $\Phi^{\oplus \cA}$ and the set of differential operators of order $r$ is precisely the image of $\Delta$.

\

\begin{rmk}\label{rmk:basicob} A few basic observations are in order.
\begin{enumerate}
\item As mentioned in the introduction, when $\pi = p$ and $n = 1$, the unique lift of Frobenius on $\ZZ_p$ is identity, thus there is only one $\d_\pi$-structure on $A_\pi$. In this case, Definition~\ref{dff1} recovers the corresponding definitions in \cite{BRS11}. In this way, the arithmetic differential operators in \cite{BRS11} are arithmetic {\it ordinary} differential operators. 
\item Arithmetic partial differential operators are closed under sums and products. 
\item The level of an analytic function is {\it not} unique. When $f$ is analytic of level $r$ it is also analytic of level $r'$ for any $r' \geq r$. At times it is convenient to say $f$ is analytic of level at least $r$. 
\end{enumerate}
\end{rmk}

\begin{xmp}
For $p$ odd, on $\ZZ_p$, the Legendre symbol gives rise to a differential operator, and thus by \cite{BRS11} an analytic function. 
\end{xmp}

\begin{xmp}
Set $A_\pi$ with two distinct $\pi$-derivations $\d_1$ and $\d_2$ which do not commute. Set $a = 0$. Fix $r \geq 2$. The function $$u \mapsto F_0(u) := \sum_{n \geq 0} \pi^{n}(\phi_1 \phi_2(u) - \phi_2 \phi_1 (u))^n$$ defines an analytic function $A_\pi \to A_\pi$ as follows. The mapping $\pi^r u \mapsto F_0(u)$ defines a function analytic at $0$ of level $r$. One can then extend this to an analytic function on $A_\pi$ by extension by zero to balls disjoint from $\pi^r A_\pi$. 
\end{xmp}

\begin{rmk}\label{rmk:expansion}
We note that $\phi_\mu$ and $\delta_\mu$ for words $\mu$ with $|\mu| \geq 2$ are not related as they are for words with $|\mu| = 1$. It is direct to verify when $\phi_1$ and $\phi_2$ are $\pi$-lifts of Frobenius, that 
$$(\phi_1 \phi_2)(x) \equiv x^{p^2} + \pi \cdot \d_1 (\pi) \cdot (\d_2 (x))^{p^2} \bmod \pi^2.$$ However, quite generally, each $\d_\mu(x)$, with $\mu = i_1 \ldots i_s$ is a polynomial expression in $\phi_{i_1}(x), \ldots, \phi_{i_s}(x)$ and their compositions. 
\end{rmk}

\section{Arithmetic differential operators are analytic} Throughout, we keep a fixed partial $\delta_{\pi}$-structure on $A_\pi$ coming from a set $\{ \phi_1,\ldots,\phi_n\}$ of Frobenius automorphisms. The main result \cite[Thm. 4, 5]{BRS11} establishes that a function $f \colon \ZZ_p \to \ZZ_p$ is an arithmetic differential operation of order at most $r$ if and only if it analytic of level $r$. The by far easier direction of the two is to show any arithmetic differential operation of level $r$ is analytic. In the language here, this calculates $\d_\mu(a+\pi^r u)$ as a power series expression, see \cite[Lem. 10]{BRS11} from which it is easy to ascertain that this power series is $\pi$-adically restricted. 

\

To generalize this direction to the PDE setting, we need the following lemma, which is an enhancement of \cite[Lem. 10]{BRS11}. However, the combinatorics needed to accurately describe $\d_\mu(a + \pi^r u)$ as a power series in the PDE setting are significantly more complicated, so we treat the lemma after developing some helpful notation. 

\

\begin{xmp}\label{xmp:n=2} It is helpful to see the complication in the case $n = 2$. A direct computation shows the shape of $\d_2(\d_1(a + \pi^r u))$. First we have from definition
\begin{eqnarray*}
\d_1( a + \pi^r u) & = & \frac{1}{\pi} \left( \phi_1(a + \pi^r u) - (a + \pi^r u)^p  \right) \\
& = & \d_1(a) + \sum_j c_j u^j + \frac{\phi_1(\pi)^r}{\pi} \phi_1(u). 
\end{eqnarray*} In this expression, each $c_j$ is divisible by $\pi^r$. As $\phi_1(\pi) \in (\pi)$ we may write this as $$\d_1(a+\pi^r u) = \d_1(a) + \pi^{r-1}  \sum_{i+j \geq 1} c(i,j) u^i \phi_1(u)^j, \textrm{ for }c(i,j) \in A_\pi.$$ 
Furthermore, expanding once again, we find an even more unwieldy expression 
{\footnotesize
\begin{eqnarray*}
\d_2\left( \d_1(a) + \pi^{r-1}  \sum c(i,j) u^i \phi_1(u)^j \right) & = & \frac{1}{\pi} \left( \phi_2\left( \d_1(a) + \pi^{r-1}  \sum c(i,j) u^i \phi_1(u)^j \right) - \left( \d_1(a) + \pi^{r-1}  \sum c(i,j) u^i \phi_1(u)^j \right)^p  \right) \\
& = &  \d_2(\d_1(a)) + \pi^{r-2}\sum_{i,j,k,\ell} c(i,j,k,\ell) u^i \phi_1(u)^j \phi_2(u)^k \phi_2(\phi_1(u))^\ell
\end{eqnarray*}} for values $c(i,j,k,\ell) \in A_\pi$. 

\end{xmp}

\

The point that Example~\ref{xmp:n=2} shows is that the expression $\d_\mu(a + \pi^m u)$ is not going to be a power series in $u$ alone, since the functions $\phi_i$ will not always be identity as was the case in \cite{BRS11}. Instead, these are complicated power series in $\phi_\nu(u)$ for all subwords $\nu \leq \mu$ and we want to introduce a more convenient notation to keep track of the terms. 

\

The key to simplifying the notation is to write these as sums over all functions from one set to another. As usual, for sets $A$ and $B$, the set of functions from $A$ to $B$ can be denoted $B^A$. For a function $\beta\in B^A$, if $B$ is a subset of the real numbers, then we set $|\beta|_1:=\sum_{a\in A}|\beta(a)|$ if $\beta(a)$ is nonzero for only finitely many inputs and $\infty$ otherwise.

\

\begin{dff}
For each $\mu \in \mathbb{M}^r$, we define the following. 
\begin{enumerate}
\item  Let $\mathcal{P}_\mu := \{ \nu \in \mathbb{M}^r \mid \nu \leq \mu \}$ be the set of subwords of $\mu$, including of course the empty subword. 
\item Set $j_{\text{const}} \colon \mathcal{P}_\mu \to \mathbf{N}$ the zero function. 
\item For $\nu\in\mathcal{P}_\mu$, the indicator function $j_\nu \colon \mathcal{P}_\mu \to \mathbf{N}$ is the function given by $j_\nu(\sigma)=0\text{ for }\sigma\neq \nu$ and $j_\nu(\nu) = 1$. 
\item Let $$J_\mu := \left\{j \colon \mathcal{P}_\mu\to \mathbf{N}\mid j(\mu)=0\text{ and } |j|_1\leq p^{|\mu|}\right\} \cup \{j_\mu\}$$ and note $j_{\text{const}} \in J_\mu$. 
\end{enumerate}
\end{dff} 

Thus, repeating and expanding the calculation from Example~\ref{xmp:n=2}, we see for each $\mu \in \mathbb{M}^r$ and $a \in A_\pi$ there exist $c_a(j,\mu)\in A_\pi$ such that \[\delta_\mu(a + \pi^r u) = \sum_{j\in J_\mu}c_a(j,\mu)\prod_{\nu\in \mathcal{P}_\mu}\phi_\nu(u)^{j(\nu)}.\] 

 \ 
Immediately, we have a relationship between arithmetic partial differential operators and analytic functions.

\begin{thm}\label{thm:PADOtoAnal} Every arithmetic partial differential operator of order at most $r$ is also analytic of level $r$. Specifically, set $\cA$ a collection of centers of balls of radius $1/\pi^r$ covering $A_\pi$ and $F\in \Gamma$. Set also $\tau_a \colon \Gamma\to \Gamma$ the $A_\pi$-algebra homomorphism induced by $\tau_a(x_\mu) := \sum_{j \in J_\mu} c_a(j,\mu) \prod_{\nu \in \mathcal{P}_\mu} x_\nu^{j(\nu)}$. We have that $\Delta(F) = \Phi^{\oplus \cA}(\tau_a(F))_{a\in \cA}$. 
\end{thm}

\begin{proof}
 Write $\delta_\mu(a+\pi^r u) = \Phi(\tau_a(x_\mu))(u)$.  Therefore, for any $a, u\in A_\pi$, 
\[	\Delta(F)(a+\pi^r u) = \Phi(\tau_a(F))(u).\qedhere\]
\end{proof}

Similar to \cite[Lem. 10]{BRS11} we estimate the $\pi$-adic absolute values of $c(j,\mu)$.

\

\begin{thm}\label{lma1} Fix $r \geq 0$. If $a \in A_{\pi}$, $v = a + \pi^r u$ in the disc $a+ \pi^r A_{\pi}$, and $\mu \in \mathbb{M}^r$, expanding, then $$\delta_{\mu}(v) = \sum_{j \in J_\mu} c_a(j,\mu) \prod_{\nu \in \mathcal{P}_\mu} \phi_\nu(u)^{j(\nu)}$$ with $c_a(j,\mu) \in A_\pi$ satisfying the following:
\begin{enumerate}
\item $|c_a(j_{\textup{const}},\mu)|_\pi \leq 1$, and in particular $c_a(j_\textup{const},\mu) = \delta_\mu(a)$,
\item $|c_a(j_{\nu},\mu)|_\pi \leq \frac{1}{\pi^{r-|\nu|}}$ for all $\nu\in\mathcal{P}_\mu$, with equality when $\nu = \mu$,
\item $|c_a(j,\mu)|_\pi \leq \frac{1}{\pi^{(r-|\mu|+1)|j|_1-1}}$ when $|j|_1\geq 2$,
\end{enumerate}
In particular, if $j\neq j_\textup{const}$, then $c_a(j,\mu)$ is a unit if and only if $j = j_\mu$ and $|\mu|= r$.
\end{thm}
\begin{proof} As $a$ is fixed throughout, we suppress this subscript from the notation. We proceed by induction on $|\mu|$. The base case $|\mu| = 0$ is immediate. Thus, we proceed with the inductive case. Assume that the estimates hold for $\mu$. Fix $1 \leq i \leq n$. We show the estimates hold for $i \cdot \mu$. By definition

\begin{align*}
	\delta_{i\cdot \mu}(a + \pi^ru)
	&= \dfrac{1}{\pi}\left(\phi_i\left(\sum_{j\in J_\mu} c(j,\mu)\prod_{\nu\in P_\mu} \phi_\nu(u)^{j(\nu)}\right) - \left(\sum_{j\in J_\mu} c(j,\mu)\prod_{\nu\in P_\mu} \phi_\nu(u)^{j(\nu)}\right)^p\right)\\
	&=\dfrac{1}{\pi}\left(\phi_i(c(j_{\mathrm{const}},\mu)) + \phi_i(c(j_\mu,\mu))\phi_{i\cdot\mu}(u) + \sum_{j\in J_\mu\colon j\neq j_{\mu,j_{\mathrm{const}}}}\phi_i(c(j,\mu))\prod_{\nu\in P_\mu} \phi_{i\cdot\nu}(u)^{j(\nu)} \right.\\
	&\qquad - \left.\left(c(j_{\mathrm{const}},\mu) + c(j_\mu,\mu)\phi_{\mu}(u) + \sum_{j\in J_\mu\colon j\neq j_{\mu,j_{\mathrm{const}}}}c(j,\mu)\prod_{\nu\in P_\mu} \phi_{\nu}(u)^{j(\nu)}\right)^p\right)
\end{align*}

\

Statement (1) and the $\nu = \mu$ case of statement (2) follow by direct computation as 
\[c(j_{\textup{const}},i \cdot \mu) = \dfrac{\phi_i(c(j_{\mathrm{const}},\mu)) -c(j_{\mathrm{const}},\mu)^p}{\pi} = \d_i(c(j_{\mathrm{const}},\mu)) = \d_i(\d_\mu(a)) = \d_{i \cdot \mu}(a)\] and 
$$c(j_{i \cdot \mu},i \cdot \mu ) = \frac{1}{\pi}\phi_i(c(j_\mu,\mu)) =  \frac{1}{\pi} \phi_i(\phi_\mu(\pi^r)/\pi^{|\mu|})$$ which forces $|c(j_{i \cdot \mu},i \cdot \mu )|_\pi = \left| \frac{1}{\pi^{|\mu|+1}} \phi_{i \cdot \mu}(\pi^r) \right|_\pi = \frac{1}{\pi^{r - |\mu| - 1}} =\frac{1}{\pi^{r - |i \cdot \mu|}}$.
 
\

For the remainder of (2), let $\nu\in\mathcal{P}_{i\cdot \mu}\setminus \{i\cdot \mu\}$.  If $\nu =i\cdot  \nu'$ and $\nu'\in \mathcal{P}_\mu$, then $c(j_\nu,i\cdot \mu)$ has as a summand $\frac{1}{\pi}\phi_i(c(j_{\nu'}, \mu))$.  By the inductive hypothesis, 
\[\left|\frac{1}{\pi}\phi_i(c(j_{\nu'}, \mu))\right|_\pi  = \pi\left|c(j_{\nu'}, \mu)\right|_\pi\leq \frac{1}{\pi^{r-|\nu'| - 1}} = \frac{1}{\pi^{r-|\nu|}}.\]
If $\nu\in \mathcal{P}_\mu$, then $c(j_\nu,i\cdot \mu)$ has as a summand $\frac{p}{\pi}c(j_\textup{const})^{p-1}c(j_\nu,\mu)$, and by the inductive hypothesis,
\[\left|\frac{p}{\pi}c(j_\textup{const})^{p-1}c(j_\nu,\mu)\right|_\pi \leq |c(j_\nu,\mu)|_\pi \leq \frac{1}{\pi^{r-|\nu|}}.\]

Now, for assertion (3), suppose that $j\in J_{i \cdot \mu}$ and $|j|_1\geq 2$. All summands of $c(j,i \cdot \mu)$ have only two forms. One form is \begin{equation*}
\frac{c(j^{(1)},\mu) \cdots c(j^{(p)},\mu)}{\pi}\end{equation*}
for some $j^{(1)},\ldots,j^{(p)} \in J_\mu$ with $\sum_{k=1}^p |j^{(k)}|_1 = |j|_1$. Note,
 we have no control over what functions $j^{(k)}$ can be. Yet, in all cases, using the inductive hypothesis we see $|c(j^{(k)},\mu)|_\pi \leq \frac{1}{\pi^{(r-|i\cdot \mu| + 1)|j^{(k)}|_1}}$ for each $k$. Therefore,  
\[\left|\frac{c(j^{(1)},\mu) \cdots c(j^{(p)},\mu)}{\pi}\right|_\pi =\pi\prod_{k=1}^p|c(j^{(k)},\mu)|_\pi\leq \frac{\pi}{\pi^{\sum_{k=1}^p(r-|i\cdot \mu|+1)|j^{(k)}|_1}} = \frac{1}{\pi^{(r-|i\cdot \mu|+1)|j|_1-1}}.\] The other type of summand only occurs when all of the words in the support of $j$ are of the form $i\cdot \nu$. In this case, $c(j,i\cdot \mu)$ also has summands of the form $\frac{\phi_i(c(\tilde j,\mu))}{\pi}$, where $\tilde j\in J_\mu$ is defined by $\tilde j(\nu) = j(i\cdot \nu)$.  In particular, this means that $|\tilde j|_1 = |j|_1\geq 2$.  Therefore by the inductive hypothesis,
\[\left|\frac{\phi_i(c(\tilde j,\mu))}{\pi}\right|_\pi \leq \pi\left|c(\tilde j,\mu)\right|_\pi \leq \frac{1}{\pi^{(r-|\mu|+1)|\tilde j|_1 - 2}}\leq  \frac{1}{\pi^{(r-|i\cdot \mu| + 1)|j|_1 + |j|_1 -2}}\leq \frac{1}{\pi^{(r-|i\cdot \mu| + 1)|j|_1 - 1}}.\qedhere\]
\end{proof}

\section{Analytic functions are arithmetic partial differential operators}\label{sec:analtoapdo}

We now conclude with the rest of the proof of the main theorem. We also include a review the broad strokes of the proof of \cite[Thm. 5]{BRS11}, which is the $n = 1$ and $\pi = p$ case of our theorem, to help the reader. Our extension requires some extra work and is notationally much more complicated. One difference is paramount enough that it introduces a hypothesis in our case that was automatic in \cite{BRS11}. We do this first and state the theorem. 

\

\begin{thm}\label{thm:analtopado} Fix $\cF = \{ \phi_1,\ldots,\phi_n\}$ with $\phi_1 = \id$ and $r \geq 0$. If $f \colon A_\pi \to A_\pi$ is analytic of level at least $r$ then $f$ is an arithmetic partial differential operator of order at most $r$. 
\end{thm}

Before proceeding with the proof of Theorem~\ref{sec:analtoapdo}, we summarize the proof of \cite[Thm. 5]{BRS11}, in a less detailed way but so that the major steps are outlined. We also indicate the differences needed in our proof before we proceed to the detailed version. In particular, \cite[Thm. 5]{BRS11} proves for an analytic function $f \colon \ZZ_p \to \ZZ_p$ analytic of level $r$, there is a restricted power series $F \in \ZZ_p[[x_0,\ldots,x_r]]$ so that $f(u) = F(u,\d(u),\ldots,\d^r(u))$ for all $u \in \ZZ_p$. Note for us $r$ denotes order, where as this is $m$ in \cite{BRS11}. They explicate $F$ as an expansion $$F(x_0,x_1,\ldots,x_m) = \sum_{n \geq 0} \sum_\beta \cs_{\beta,n} x_0^{\beta_0}\cdots x_{r-1}^{\beta_{r-1}} x_r^n$$ and show that when $F$ is of this form, the coefficients $\cs_{\beta,n}$ are uniquely determined. The sum over $\beta$ is over tuples $\beta = (\beta_0,\ldots,\beta_{r-1})$ with each $0 \leq \beta_i \leq p-1$. In \cite{BRS11}, $\cs_{\beta,k}$ is denoted $a_{\beta,k}$, but this notational change will help resolve confusion in our argument.

\

The first step for \cite[Thm. 5]{BRS11} is to reduce to the case where $f$ is supported on a single $p$-adic ball of radius $1/p^r$ and that $f(p^r u) = u^\ell$. This case is sufficient as any analytic function is an algebraic combination of these. Next, one approximates $F$ by a $u$-adically convergent series $F^k$, $F^k \to F$ as $k \to \infty$ which are constructed inductively on $k$. In each step of the induction, the vector of coefficients $(\cs_{\beta,k})_{\beta}$ are determined by a linear system of the following form 
\begin{equation}\label{eq:infinitywars}
c_k W \cdot (\cs_{\beta,k})_{\beta} + C_{k-1} = d.
\end{equation} We explain the terms in this system. The constant $C_{k-1}$ is a term coming from the $(k-1)$-step of the induction and $c_k$ is a constant that appears as a coefficient in the expression $\d(a + p^r u)$. The matrix $W$ is a $p^r \times p^r$-matrix\footnote{Note the unfortunate typo in \cite{BRS11} on the size is corrected in \cite{Bra20}.} built entirely out of products of iterations of $\d$ at a complete family of residue centers and is independent of $f$. The value $d$ is a product of  Kronecker $\d$-functions constructed to ensure $F^k$ approximates $f$ up to order $u^{k+1}$. Now, to solve this system, two major points are established. First, the value $c_k$ is a $p$-adic unit, which follows from the $p$-adic estimates analogous to Theorem~\ref{lma1}. Second, the matrix $W$ is shown to be invertible. 

\

For the adaptation, the  
$F^k$ we introduce is much more complicated to express, however it is nearly a direct analog of what is used in \cite{BRS11}. Indeed, the series $F^k$ involves a clever choice of centers in a covering family of balls choosen via \cite[Lem. 11]{BRS11}. The most obvious analog of \cite[Lem. 11]{BRS11} is too much to hope for in the multivariate setting. Instead we utilize the assumption that $\phi_1 = \id$ to allow us to adapt the main theorem of \cite{BRS11} as a bootstrap. 
Additionally, $F$ is written in terms of monomials of the form $\phi_\mu(u)^{j(\mu)}$ for various $\mu \in \mathbb{M}^r$. Many of these functions $\phi_\mu$ as $\mu$ ranges coincide. The hypothesis that some $\phi_i = \id$ is helpful towards controlling this, but a complicated reindexing is also needed to find the right coefficients. 

\ 

 The convergence we need now works modulo powers of the `ideal' $\fm := (\phi_\mu(u) \colon \mu \in \mathbb{M}^r)$ in a fashion made precise in the proof. Once setup properly, the critical point we arrive at is a linear system similar to that of \eqref{eq:infinitywars}. Specifically, the new system has the form 
\begin{equation}\label{eq:doomsday}
(C_{k-1,\ell})_\ell + L \cdot W \cdot (\cs_{\beta,\gamma})_\beta = (d_\ell)_\ell. 
\end{equation}
The system \eqref{eq:doomsday} is of course more complicated than \eqref{eq:infinitywars}, which is apparent by its shape. To explicate, here $\gamma$ ranges over functions $\gamma \colon {\mathbb M}^{= r} \to \NN$ with $|\gamma|_1 \leq k$, whereas when $n = 1$, $\# {\mathbb M}^{=r} = 1$ and $\ell$ runs over a similar type of family of functions. This means we have a larger system to consider which is not surprising. The vector $(d_\ell)_\ell$ is chosen to ensure $F^k$ and $f$ agree up modulo $\fm^{k+1}$. The term $(C_{k-1,\ell})_\ell$ is now a vector of values coming from the $(k-1)$-step of the induction. This time, $W$ is {\it not} a square matrix, but instead has $p^r$ rows and a column for every function $\beta \colon \mathbb{M}^{r-1} \to \NN$ with $\beta(\mu) < p$ for all $\mu$. In the case $\pi = p$ and $n = 1$, $\# \mathbb{M}^{r-1} = r$ yielding the square shape of the $W$ matrix in \cite{BRS11}. We show however that this expanded $W$ still has `full rank', see Lemma~\ref{lem:W}. The matrix $L$ arises roughly from the regrouping needed to extract the coefficient of $\phi_\eta$ as $\eta$ runs only over a set of words including each unique function $\phi_\mu$ once. It plays the role of $c_k$ before. This matrix $L$ is also not square, thus it is impossible to show it is invertible. Instead, we also show that the above system can be solved. Note the systematic use of non-square matrices means one cannot expect uniqueness to hold. 

\

Now, to begin working towards the proof of Theorem~\ref{thm:analtopado}, we select a collection of centers of $\pi$-adic balls of radius $1/\pi^r$ with certain properties. We make this selection because generally, we see no analog \cite[Lem. 11]{BRS11} to hold for partial $\d$-structures without imposing conditions relating the associated lifts of Frobenius. 

\begin{lem}\label{lem:ODE} Let $\phi = \id$, $\delta$ the corresponding $\pi$-derivation, and $r\geq 1$.  There exists a set $\cA\subset A_\pi$ with $\#\cA = p^r$ such that $\cA + \pi^rA_\pi = A_\pi$ and, for all $a\in \cA$, $\delta^r(a) = 0$.
\end{lem}
\begin{proof} This is done by inductively constructing sets $C_i$ for $1 \leq i \leq r$ for which $\# C_i = p^i$ and ensuring the property $\d^i(a) = 0$ for all $a \in C_i$. One then sets $\cA := C_r$. Indeed, $C_1 = \{ a \in \ZZ_p \colon a^p = a\}$ is the set of constant\footnote{Normally, these are referred to Teichm\"uller lifts, but owing to the Nazi affiliations of O.\ Teichm\"uller, we have opted to refer to these as constant lifts.} lifts. To construct $C_i$, for each $a \in C_{i-1}$, use Hensel lifting on the polynomial $t^p - t + \pi a \in A_\pi[t]$ to find a set of $p$ roots which forms part of $C_i$. Furthermore, these roots are all distinct modulo $\pi$.  This ensures by induction that $\d^i(a) = 0$ for all $a \in C_i$ and moreover $\# C_i = p^i$ by construction. 

Finally, if $a, a'\in C_i$ are distinct and $a \equiv a' \mod \pi^i$, then $\delta(a)\neq \delta(a')$ since $a$ and $a'$ come from roots of different polynomials in the previous step.  However, we also have that $(a - a^p)/\pi\equiv (a'-(a')^p)/\pi \mod \pi^{i-1}$. So we have $\delta(a)\equiv \delta(a') \mod \pi^{i-1}$, a contradiction to the inductive hypothesis.  Therefore the natural map $C_r \subset A_\pi \to A_\pi/ \pi^r \cong \FF_p^r$ is injective and thus a bijection. 
\end{proof}

\

Continuing towards the proof of Theorem~\ref{thm:analtopado}, we handle the analog of the matrix $W$. We appeal to linear algebra over the ring $A_\pi$. As such, we abuse notation and refer to a matrix as {\bf full rank} provided it has a maximal minor with unit determinant. We warn that this is a stronger condition than being full rank after changing to the fraction field of $A_\pi$. We first prove the needed rank condition analogous to \cite[Lem. 12]{BRS11}.

\begin{lem}\label{lem:W} For fixed $r \geq 1$, set $\mathcal{A}$ as in Lemma~\ref{lem:ODE} and $$\mathfrak{B}:= \left\{ \beta \colon \mathbb{M}^{r-1} \to \NN \textrm{ so that } \beta(\mu) < p \textrm{ for each } \mu \in \mathbb{M}^{r-1}\right \}.$$ The $p^r \times \# \mathfrak{B}$-matrix $W := (w_{a\beta})_{a \in \mathcal{A}, \beta \in \mathfrak{B}}$ where $$w_{a\beta} = \prod_{\mu\in \mathbb{M}^{r-1}}\delta_\mu(a)^{\beta(\mu)}$$ has full rank.
\end{lem}
\begin{proof}
\noindent Note the elements of $\cA$ are centers of $p^r$ balls of radius $1/\pi^r$ which cover $A_\pi$. We exhibit a selection of columns for which the $p^r \times p^r$-minor of $W$ has unit determinant. Fix $i \in \{ 1,\ldots,n\}$ and let $B$ be the collection of maps $\beta\in \mathfrak{B}$ with support in the words $i^j := \underbrace{ii\cdots i}_{j\textrm{ times}}$ for $0 \leq j \leq r-1$.
We have $\# B=p^r$ and so the submatrix of $W$ with columns corresponding to the maps in $B$ is a square matrix. Consider the matrix $V=(v_{a\beta})_{\beta \in B}$ with entries in $\FF_p$ given by $v_{a\beta} = \overline{w_{a\beta}}$, where $\overline{w_{a\beta}}$ is the image of $w_{a\beta}$ under the projection map $A_\pi\to A_\pi/\pi A_\pi\cong \FF_p$. We utilize a variant of the map $\nabla$ defined in \cite[Eq. 3.8, Sec. 3.3, pg. 78]{Bui05}. Specifically, for each $1 \leq i \leq n$, set $\nabla_i^r \colon A_\pi \to A_\pi^r$ the map sending $v \mapsto (v, \d_i(v), \ldots, \d_i^{r-1}(v))$. By composing with reduction modulo $\pi$ we consider $\overline{\nabla}^r_i \colon A_\pi/\pi^rA_\pi \to \FF_p^r$. A direct adaptation of \cite[Lem. 3.20]{Bui05} shows $\overline{\nabla}^r_i$ is a surjection for each $i$. So for each $\mathbf{v} = (v_0,v_1,\ldots, v_{r-1})\in \FF_p^r$, there exists $a \in \mathcal{A}$ such that
$$ \mathbf{v} = \overline{\nabla}^r_i(a) = (\overline{a}, \overline{\delta_i(a)}, \cdots, \overline{\delta_i^{r-1}(a)}).$$
Therefore, for any $\beta\in B$, 
$$\mathbf{v}^\beta := v_0^{\beta(i^0)}v_1^{\beta(i^1)}\cdots v_{r-1}^{\beta(i^{r-1})} = \overline{a}^{\beta(i^0)} \overline{\delta_i(a)}^{\beta(i^1)} \cdots \overline{\delta_i^{r-1}(a)}^{\beta(i^{r-1})} = \overline{w_{a\beta}}.
$$

Thus, for every $\mathbf v\in \FF_p^r$, the row vector $(\mathbf{v}^\beta)_{\beta\in B}$ is a row of $V$.  Let $(b_\beta)_{\beta\in B} \in \FF_p^{\# B}$ be a vector in the nullspace of $V$, i.e.\ for all 
$\textbf{v}\in \FF_p^r$, $\sum_{\beta\in B}\textbf{v}^\beta b_\beta = 0$. An easy argument then shows $b_\beta = 0$ for all $\beta \in B$.  Indeed, we have a polynomial $\sum_{\beta \in B} b_\beta X^\beta$ in $r$ variables, each variable of degree at most $p - 1$, but which has $p^r$ roots. The only such polynomial is the zero polynomial.
Therefore $\det(V)\neq 0$, and so the determinant of the submatrix of $W$ with columns from $B$ is a unit.  Thus the claim holds.
\end{proof}

For $\mu,\nu\in \mathbb M^r$, let $\mu\sim \nu$ if $\phi_\mu = \phi_\nu$. For example, when $\phi_1 = \id$, every word is equivalent to a word of length $r$. This condition plays a role in Theorem~\ref{thm:gutsofproof}. Let $E=\{\eta_1,\ldots, \eta_t\}\subset \mathbb M^{r}$ be a complete set of representatives of the equivalence classes of $\sim$; that is, every $\mu\in \mathbb M^r$ is equivalent to exactly one of the words in $E$.  We have that $\Gamma_E = A_\pi[[x_{\eta}\colon \eta\in E]]\cap \Gamma$ is the set of restricted power series in the $x_\eta$ for $\eta\in E$, so let $\cE:\Gamma\to \Gamma_E$ be the surjective $A_\pi$ algebra homomorphism induced by $\cE(x_\mu) := x_{\eta}$, where $\eta$ is the unique element of $E$ such that $\mu \sim \eta$.  Note that $\Phi_E\circ \cE = \Phi$.  

\ 

The following diagram illustrates the relationships between the various sets involved. Recall, $\tau_a \colon \Gamma\to \Gamma$ is the $A_\pi$-algebra homomorphism induced by $\tau_a(x_\mu) := \sum_{j \in J_\mu} c_a(j,\mu) \prod_{\nu \in \mathcal{P}_\mu} x_\nu^{j(\nu)}$. The arithmetic partial differential operators are those functions in the image of $\Delta$, and the analytic functions are the functions in the image of $\Phi^{\oplus \cA}$. The right triangle obviously commutes and the left triangle commutes by Theorem~\ref{thm:PADOtoAnal}. Clearly the image of $\Phi^{\oplus \cA}$ and $\Phi_E^{\oplus \cA}$ agree. The content of the main theorem is that the images of $\Delta$ and $\Phi^{\oplus \cA}$ are equal, thus it suffices to show the composition $\cN$ along the horizontal maps is surjective.

\begin{center}
\begin{tikzpicture}
	\node (Gamma) at (0,0) {$\Gamma$};
	\node (Fun) at (3,-2) {$\mathrm{Fun}(A_\pi,A_\pi)$};
	\node (GammaA) at (3,0) {$\Gamma^{\oplus \cA}$};
	\node (GammaAE) at (6,0) {$\Gamma_E^{\oplus \cA}$};
	\draw[->] (Gamma) -- (Fun) node[midway, below left] {$\Delta$};
	\draw[->] (GammaA) -- (Fun) node[midway, right] {$\Phi^{\oplus \cA}$};
	\draw[->] (Gamma) -- (GammaA) node[midway, below] {$\bigoplus_{a\in \cA}\tau_a$};
	\draw[->] (GammaA) -- (GammaAE) node[midway, below] {$\cE^{\oplus \cA}$};
	\draw[->] (GammaAE) -- (Fun) node[midway, below right] {$\Phi_E^{\oplus \cA}$};
	\draw[->] (Gamma) to[bend left] node[above, midway]{$\mathcal{N}$} (GammaAE) ;
\end{tikzpicture}
\end{center}

As we have no general analog of \cite[Lem. 11]{BRS11}, our construction suffers from one more issue, namely we cannot guarantee $\d_\mu(a) = 0$ for a collection of centers of a covering family of balls of radius $1/\pi^r$. Instead, we articulate a sufficient condition to ensure analytic functions are arithmetic partial differential operators. After this, we will utilize this condition first to show locally constant functions are  arithmetic partial differential operators and then deduce the general case from this. We denote by $\cN := \cE^{\oplus \cA} \circ \bigoplus_{a \in \cA} \tau_a$.

\begin{thm} \label{thm:gutsofproof} Suppose $(\delta_1,\ldots, \delta_n)$ is a partial $\delta$-structure, $r\geq 0$, and that every word in $\mathbb M^r$ is equivalent to a word of length $r$.  If for each $\mu\in \mathbb M^r$ there exists $K^\mu\in \Gamma$ such that $\cN(K^\mu) = (\delta_\mu(a))_{a\in \cA}$, then the map $\cN: \Gamma \to \Gamma_E^{\oplus \cA}$ is surjective.
\end{thm}

\begin{proof}

	 We first establish some notation. In addition to retaining the notation for $\mathcal{A}$ and $\mathfrak{B}$ in Lemma~\ref{lem:W}, we define the following sets
$$\NN_{k}^{\mathbb{M}^{=r}} := \left\{ \gamma \colon \mathbb{M}^{=r} \to \NN \textrm{ so that } |\gamma|_1 \leq k\right \} \textrm{ and } 
\NN_{=k}^{\mathbb{M}^{=r}} := \left\{ \gamma \colon \mathbb{M}^{=r} \to \NN \textrm{ so that } |\gamma|_1 = k\right \}.$$
By hypothesis we may also assume that $E\subset \mathbb{M}^{=r}$.

\ 

Let $(F_a)_{a\in\cA}\in \Gamma_E^{\oplus \cA}$.  We will construct $F\in \Gamma$ such that $\cN(F) = (F_a)_{a\in\cA}$ by inductively defining $s_{\beta,\gamma}$ for $\beta\in \mathfrak{B}$, $\gamma\in \NN_{k}^{\mathbb{M}^{=r}}$ so that
\[F := \sum_{\gamma\in \NN^{\mathbb{M}^{=r}}}\sum_{\beta \in \mathfrak{B}} \cs_{\beta,\gamma} \prod_{\mu \in \mathbb{M}^{r-1}} x_\mu^{\beta(\mu)}\prod_{\mu \in \mathbb{M}^{=r}} (x_\mu - K^\mu)^{\gamma(\mu)},\]
where $K^\mu\in \Gamma$ such that $\cN(K^\mu) = (\delta_\mu(a))_{a\in \cA}$.
For $k\in \NN$ we define $F^k\in \Gamma$ to be
\[F^k := \sum_{\gamma\in \NN_k^{\mathbb{M}^{=r}}}\sum_{\beta \in \mathfrak{B}} \cs_{\beta,\gamma} \prod_{\mu \in \mathbb{M}^{r-1}} x_\mu^{\beta(\mu)}\prod_{\mu \in \mathbb{M}^{=r}} (x_\mu - K^\mu)^{\gamma(\mu)}\]
 so that $F^k\to F$ as $k\to \infty$. 

\ 

Set $\fm := (x_\eta \colon \eta \in E)\subset \Gamma_E$, and we inductively construct $s_{\beta,\gamma}$ such that for all $k\in\NN$ and $a\in\cA$, $(\cE\circ \tau_a)(F^k) \equiv F_a \mod \fm^{k+1}$, that is, the coefficients of all terms of degree $k$ or less in $(\cE\circ \tau_a)(F^k)$ and $F_a$ agree.  

\ 

We first handle the base case of the induction $k=0$. We want, for all $a\in\cA$, 
$$(\cE\circ\tau_a)(F^0) \equiv F_a\mod \fm$$
Now modulo $\fm$, and since $\cE$ preserves degree and is identity on $A_\pi$,
\begin{align*}
	(\cE\circ\tau_a)(F^0) & = \sum_{\beta \in \mathfrak{B}} \cs_{\beta,0} \prod_{\mu \in \mathbb{M}^{r-1}} (\cE\circ \tau_a)(x_\mu)^{\beta(\mu)}\\ 
	& = \sum_{\beta \in \mathfrak{B}} \cs_{\beta,0} \prod_{\mu \in \mathbb{M}^{r-1}} \cE\left(\sum_{j\in J_\mu} c_a(j,\mu)\prod_{\nu\in\mathcal{P}_\mu}x_{\nu}^{j(\nu)}\right)^{\beta(\mu)}\\ 
	&\equiv \sum_{\beta \in \mathfrak{B}} \cs_{\beta,0} \prod_{\mu \in \mathbb{M}^{r-1}}c_a(j_\textup{const},\mu)^{\beta(\mu)} \mod \fm \\ 
	&\equiv \sum_{\beta \in \mathfrak{B}} \cs_{\beta,0} \prod_{\mu \in \mathbb{M}^{r-1}} \delta_\mu(a)^{\beta(\mu)}\mod \fm \\
	&\equiv \sum_{\beta \in \mathfrak{B}} \cs_{\beta,0}w_{a,\beta} \mod \fm
\end{align*}

Thus, we seek coefficients $\cs_{\beta,0}$, $\beta\in \mathfrak{B}$, such that for all $a \in \mathcal{A}$, $\sum_{\beta \in \mathfrak{B}} \cs_{\beta,0} w_{a \beta} = d_a$, where $d_a$ is the constant term of $F_a$.  This is equivalent to the matrix equation 
\begin{equation*}
W(\cs_{\beta,0})_{\beta \in \mathfrak{B}} = (d_a)_{a\in \mathcal{A}}.
\end{equation*}
Since $W$ is of full rank by Lemma~\ref{lem:W}, such a vector $(\cs_{\beta,0})$ exists, and the base case is proved.

\

Now we handle the inductive step. Suppose that $k\geq 1$ and that for all $a\in\cA$, $(\cE\circ \tau_a)(F^{k-1}) \equiv F_a \mod \fm^k$.  We wish to construct $\cs_{\beta,\gamma}$ for $\beta\in\mathfrak{B}$ and $\gamma\in \NN_{=k}^{\mathbb{M}^{=r}}$ so that similar equivalences hold modulo $\fm^{k+1}$. 

\

For $a\in \cA$, $\beta\in\mathfrak{B}$, and $\gamma\in \NN_k^{\mathbb{M}^{=r}}$, set 
\[G_{a,\beta} := \prod_{\mu \in \mathbb{M}^{r-1}} \left((\cE\circ \tau_a)(x_\mu)\right)^{\beta(\mu)}\quad \text{and}\quad H_{a,\gamma} := \prod_{\mu \in \mathbb{M}^{=r}}\left((\cE\circ \tau_a)(x_\mu - K^\mu)\right)^{\gamma(\mu)},\] so that
\begin{align*}(\cE\circ \tau_a)(F^k) &= \sum_{\gamma\in \NN_{k}^{\mathbb{M}^{=r}}}\sum_{\beta \in \mathfrak{B}} \cs_{\beta,\gamma}G_{a,\beta}H_{a,\gamma}\\
& = \sum_{\gamma\in \NN_{<k}^{\mathbb{M}^{=r}}}\sum_{\beta \in \mathfrak{B}} \cs_{\beta,\gamma}G_{a,\beta}H_{a,\gamma} + \sum_{\gamma\in \NN_{=k}^{\mathbb{M}^{=r}}}\sum_{\beta \in \mathfrak{B}} \cs_{\beta,\gamma}G_{a,\beta}H_{a,\gamma} \\
&= (\cE\circ \tau_a)(F^{k-1}) + \sum_{\gamma\in \NN_{=k}^{\mathbb{M}^{=r}}}\sum_{\beta \in \mathfrak{B}} \cs_{\beta,\gamma}G_{a,\beta}H_{a,\gamma}.
\end{align*}
For any $\mu$, the constant term of $(\cE\circ \tau_a)(x_\mu - K_\mu)$ is $c_a(j_\textup{const},\mu) - (\cE\circ \tau_a)(K^\mu) = \delta_\mu(a)- \delta_\mu(a) = 0$.  Therefore $H_{a,\gamma}\in \fm^k$ for all $\gamma\in \NN_{=k}^{\mathbb{M}^{=r}}$. 

\ 

Furthermore, the degree $k$ component of $H_{a,\gamma}$ is the product of the degree 1 components of the factors \newline $(\cE\circ \tau_a)(x_\mu - K_\mu)$.  Since $(\cE\circ \tau_a)(K_\mu)$ is constant, we need only consider the degree 1 components of  $(\cE\circ \tau_a)(x_\mu)$.  For $a\in\cA$, $j\in \NN^E$, let $c_a'(j,\mu)\in A_\pi$ such that
\[(\cE\circ \tau_a)(x_\mu) = \sum_{j\in \NN^E} c_a'(j,\mu) \prod_{\eta\in E}x_\eta^{j(\eta)}.\]
The degree 1 terms are precisely those corresponding to the indicator functions $j_\eta$ for $\eta\in E$.  Therefore, 
\[(\cE\circ \tau_a)(x_\mu - K_\mu)\equiv \sum_{\eta\in E}c_a'(j_\eta,\mu)x_\eta \mod \fm^2.\]
Thus, we have that, when $|\gamma|_1 = k$,
\[H_{a,\gamma} \equiv \prod_{\mu\in \mathbb{M}^{=r}}\left(\sum_{\eta\in E}c_a'(j_\eta,\mu)x_\eta\right)^{\gamma(\mu)} \mod \fm^{k+1}.\]

\ 

This implies that modulo $\fm^{k+1}$, $G_{a,\beta}H_{a,\gamma}$ is congruent to the constant term of $G_{a,\beta}$ times $H_{a,\gamma}$.  We have that $(\cE\circ \tau_a)(x_\mu) \equiv c_a(j_\textup{const},\mu) \equiv \delta_\mu(a) \mod \fm$, and therefore $G_{a,\beta}\equiv \prod_{\mu \in \mathbb{M}^{r-1}} \delta_\mu(a)^{\beta(\mu)}\equiv w_{a,\beta}\mod \fm$.  Hence
\begin{align*}
	(\cE\circ\tau_a)(F^k) \equiv (\cE\circ \tau_a)(F^{k-1}) + \sum_{\gamma\in \NN_{=k}^{\mathbb{M}^{=r}}}\sum_{\beta \in \mathfrak{B}} \cs_{\beta,\gamma}w_{a,\beta}\prod_{\mu\in \mathbb{M}^{=r}}\left(\sum_{\eta\in E}c_a'(j_\eta,\mu)x_\eta\right)^{\gamma(\mu)} \mod \fm^{k+1}
\end{align*}

Now since all terms in the second summand on the right above have degree $k$, the only terms of degree less than $k$ come from $(\cE\circ \tau_a)(F^{k-1})$, which by the inductive hypothesis are congruent to $F_a$ modulo $\fm^k$ for all $a\in \cA$.  Thus we need only ensure that the coefficients of the degree $k$ monomials in $(\cE\circ \tau_a)(F^k)$ and $F_a$ are equal.

\ 

For each $a\in \cA$, $\ell\in \NN^E$ with $|\ell|_1=k$, and $\gamma\in \NN_{=k}^{\mathbb M^r}$, let $L_{a,\ell,\gamma}$ be the coefficient in multidegree $\ell$ in $\prod_{\mu\in \mathbb{M}^{=r}}\left(\sum_{\eta\in E}c_a'(j_\eta,\mu)x_\eta\right)^{\gamma(\mu)}$, i.e.\ the coefficient of $\prod_{\eta\in E}x_\eta^{\ell(\eta)}$.  Let $C_{a,\ell}^{k-1}$ be the coefficient in multidegree $\ell$  in $(\cE\circ \tau_a)(F^{k-1})$ and $d_{a,\ell}$ the coefficient in multidegree $\ell$ of $F_a$.  Thus, we wish to find $s_{\beta,\gamma}$ such that for all $a$ and $\ell$,
\[C_{a,\ell}^{k-1} +  \sum_{\gamma\in \NN_{=k}^{\mathbb{M}^{=r}}}\sum_{\beta \in \mathfrak{B}} \cs_{\beta,\gamma}w_{a,\beta}L_{a,\ell,\gamma} = d_{a,\ell}.
\]
Let $S_{a,\gamma} = \sum_{\beta \in \mathfrak{B}} \cs_{\beta,\gamma}w_{a,\beta}$, so that the system of equations above may be rewritten
\[\text{for all }a\in \cA\text{ and all }\ell\in \NN^E\text{ with }|\ell|_1 = k, \quad \sum_{\gamma\in \NN_{=k}^{\mathbb{M}^{=r}}} L_{a,\ell,\gamma}S_{a,\gamma} = d_{a,\ell} - C^{k-1}_{a,\ell}\]
or equivalently as the system of matrix equations
\[\text{for all }a\in \cA, \quad (L_{a,\ell,\gamma})_{\ell,\gamma}(S_{a,\gamma})_\gamma  = (d_{a,\ell} - C^{k-1}_{a,\ell})_\ell\]
We claim that for each $a\in \cA$, the matrix $L_a:=(L_{a,\ell,\gamma})_{\ell,\gamma}$ has full rank, and prove it by showing that for each $\ell\in\NN^E$ with $|\ell|_1=k$, we may select $\gamma_\ell\in \NN_{=k}^{\mathbb M^r}$ such that the square submatrix of $L_a$ with the columns $\gamma_\ell$ has unit determinant.  

\ 

For each $\ell\in\NN^E$ with $|\ell|_1=k$, let $\gamma_\ell\in \NN_{=k}^{\mathbb M^r}$ be defined by $\gamma_\ell(\mu) = \ell(\mu)$ if $\mu\in E$ and $\gamma_\ell(\mu) = 0$ otherwise.  Now, each entry of the submatrix of $L_a$ with columns $\gamma_\ell$ is $L_{a,\ell',\gamma_\ell}$ for some $\ell$, $\ell'$, which is the coefficient in multidegree $\ell'$ in the expression 
\[\prod_{\mu\in \mathbb{M}^{=r}}\left(\sum_{\eta\in E}c_a'(j_\eta,\mu)x_\eta\right)^{\gamma_\ell(\mu)} = \prod_{\mu\in E}\left(\sum_{\eta\in E}c_a'(j_\eta,\mu)x_\eta\right)^{\ell(\mu)}\]

\ 

Recall that $c_a'(j_\eta,\mu)$ is the coefficient of $x_\eta$ in $(\cE\circ\tau_a)(x_\mu)$.  Since $\cE$ preserves degree, the coefficient of $x_\eta$ in $(\cE\circ\tau_a)(x_\mu)$ is the sum of the coefficients of $x_\nu$ in $\tau_a(x_\mu)$ with $\nu\sim \mu$, that is, $c_a'(j_\eta,\mu) = \sum_{\nu\in \cP_\mu\colon \nu\sim \eta} c_a(j_\nu,\mu)$.  By Theorem~\ref{lma1}, $c_a(j_\nu,\mu)$ is a unit if and only if $\nu = \mu$ and $|\mu| = r$.  Therefore, for $\eta,\mu\in E$, $c'_a(j_\eta,\mu)$ is a unit if and only if $\eta\sim \mu$, i.e.\ $\eta = \mu$.  Therefore, the only term of $\prod_{\mu\in E}\left(\sum_{\eta\in E}c_a'(j_\eta,\mu)x_\eta\right)^{\ell(\mu)}$ with a unit coefficient is $\prod_{\eta\in E}c_a'(j_\eta,\eta)^{\ell(\eta)}x_\eta^{\ell(\eta)}$, i.e.\ the one with multidegree $\ell$.

\ 

Thus, we've shown that $L_{a,\ell',\gamma_\ell}$ is a unit if and only if $\ell' = \ell$.  By Gaussian elimination, the maximal minor of $L_a$ with columns $\gamma_\ell$ has unit determinant.  Therefore, for each $a\in\cA$ we may solve the system $(L_{a,\ell,\gamma})_{\ell,\gamma}(S_{a,\gamma})_\gamma  = (d_{a,\ell} - C^{k-1}_{a,\ell})_\ell$ for the vector $(S_{a,\gamma})_\gamma$.   Finally we need to find $\cs_{\beta,\gamma}$ satisfying $\sum_{\beta\in \mathfrak{B}} \cs_{\beta,\gamma}w_{a,\beta} = S_{a,\gamma}$.  For each $\gamma\in\NN_{= k}^{\mathbb{M}^{=r}}$, we must solve the matrix equation 
$W(\cs_{\beta,\gamma})_{\beta \in \mathfrak{B}}= (S_{a,\gamma})_{a \in \mathcal{A}}$, which is possible since $W$ is of full rank by Lemma~\ref{lem:W}.

\ 

To show that the power series $F$ is restricted, we show that the $\pi$-adic order of the coefficients $\cs_{\beta,\gamma}$ decreases as $|\gamma|_1$ increases.  For a vector $(\alpha_i)$ of elements of $A_\pi$, let $|(\alpha_i)|_\pi = \max_i(|\alpha_i|_\pi)$.  We use the fact that if $M$ is a matrix over $A_\pi$ of full rank and $\mathbf{x}$, $\mathbf{y}$ are vectors with entries in $A_\pi$, and $M\mathbf{x}=\mathbf{y}$, then $|\mathbf{x}|_\pi = |\mathbf{y}|_\pi$.

\ 

Let $t_0=1$ and for $k\geq 1$ set $D_k := \max\{ |d_{a,\ell}|_\pi\colon a\in\mathcal{A}, \ell\in \NN^E\text{ with }|\ell|_1=k\}$ and $t_k := \max\{t_{k-1}/\pi, D_k\}$.  We claim that $|s_{\beta,\gamma}|_\pi \leq t_k$ for all $\gamma$ with $|\gamma|_1=k$.  This will suffice to show that $F$ is restricted.  
When $k=0$ this is obvious.  Suppose that $k\geq 1$ and the statement is shown for $|\gamma|_1 = k-1$.  For each $\gamma$ with $|\gamma|_1=k$ we have that $W(\cs_{\beta,\gamma})_{\beta \in \mathfrak{B}}= (S_{a,\gamma})_{a \in \mathcal{A}}$, and $W$ is of full rank, and so $|(\cs_{\beta,\gamma})_{\beta \in \mathfrak{B}}|_\pi = |(S_{a,\gamma})_{a \in \mathcal{A}}|_\pi$.  Now, for each $a\in\mathcal{A}$, we have that $(L_{a,\ell,\gamma})_{\ell,\gamma}(S_{a,\gamma})_{\gamma\in \NN^{\mathbb{M}^{=r}}_{=k}} = (d_{a,\ell} - C_{a,\ell}^{k-1})_{\ell}$, and since the matrix $(L_{a,\ell,\gamma})$ has full rank, 
\[|(S_{a,\gamma})_{\gamma\in \NN^{\mathbb{M}^{=r}}_{=k}}|_\pi = |(d_{a,\ell} - C_{a,\ell}^{k-1})_{\ell}|_\pi = \max_\ell(|d_{a,\ell}-C_{a,\ell}^{k-1}|_\pi) \leq \max_{\ell}(|d_{a,\ell}|_\pi, |C_{a,\ell}^{k-1}|_\pi)\leq \max(t_k,\max_{\ell}|C_{a,\ell}^{k-1}|).\]
Thus, it suffices now to show that $|C_{a,\ell}^{k-1}|\leq \frac{t_{k-1}}{\pi}$ for all $a\in\cA$ and $\ell\in \NN^E$ with $|\ell|_1=k$.

\ 

Since $C_{a,\ell}^{k-1}$ is the coefficient of $\prod_{\eta\in E}x_\eta^{\ell(\eta)}$ in $F_a^{k-1}$, it is a sum of terms of the form
\[s_{\beta,\gamma}\left(\prod_{\mu\in \mathbb{M}^{r-1}}\prod_{i=1}^{\beta(\mu)}c_a(j^{(\mu,i)},\mu)\right)\left(\prod_{\mu\in \mathbb{M}^{=r}}\prod_{i=1}^{\gamma(\mu)}c_a(j^{(\mu,i)},\mu)\right)\]
where $|\gamma|_1\leq k-1$, $j^{(\mu,i)}\in J_{\mu}$ for each $\mu$ and $i$, and  $|\ell|_1 = \sum_{\mu\in \mathbb{M}^{r-1}}\sum_{i=1}^{\beta(\mu)}|j^{(\mu,i)}|_1 + \sum_{\mu\in \mathbb{M}^{=r}}\sum_{i=1}^{\gamma(\mu)}|j^{(\mu,i)}|_1$. 
By Theorem~\ref{lma1}, for any $\mu\in\mathbb{M}^{r-1}$ and $j\in J_\mu$, $|c_a(j,\mu)|_\pi\leq \frac{1}{\pi^{|j|_1}}$, and for any $\mu\in \mathbb{M}^{=r}$ and $j\in J_\mu$, $|c_a(j,\mu)|_\pi\leq \frac{1}{\pi^{|j|_1 - 1}}$.
Also,
\begin{align*}\sum_{\mu\in \mathbb{M}^{r-1}}\sum_{i=1}^{\beta(\mu)}|j^{(\mu,i)}|_1 + \sum_{\mu\in \mathbb{M}^{=r}}\sum_{i=1}^{\gamma(\mu)}\left(|j^{(\mu,i)}|_1 - 1\right) 
&= \sum_{\mu\in \mathbb{M}^{r-1}}\sum_{i=1}^{\beta(\mu)}|j^{(\mu,i)}|_1+ \sum_{\mu\in \mathbb{M}^{=r}}\sum_{i=1}^{\gamma(\mu)}|j^{(\mu,i)}|_1  - \sum_{\mu\in \mathbb{M}^{=r}}\gamma(\mu) \\
& \geq  |\ell|_1 - (k-1) = 1.\end{align*} By the induction hypothesis, $|s_{\beta,\gamma}|_\pi \leq t_{k-1}$. Putting these together, we see $|C_{a,\ell}^{k-1}|_\pi \leq \frac{t_{k-1}}{\pi}$. Thus $|s_{\beta,\gamma}|_\pi\leq t_k$ when $|\gamma|_1=k$ as desired. The terms $t_k \to 0$  $\pi$-adically as $k \to {}\infty$, thus $F$ is restricted.
\end{proof}

The proof of Theorem~\ref{thm:analtopado} now comes down to showing that the hypotheses of Theorem~\ref{thm:gutsofproof} hold.

\begin{proof}[Proof of Theorem~\ref{thm:analtopado}] We first show how to construct $K^\mu$ as needed in Theorem~\ref{thm:gutsofproof}. Specifically, note that as $\phi_1 = \id$ we may view in a distinguished way $\mathbb{M}_1 \subset \mathbb{M}_n$. This means an arithmetic partial differential operator in the variables corresponding to $\mathbb{M}_1$ is still an arithmetic partial differential operator, so we may without loss of generality construct $K^\mu$ using $\mathbb{M}_1$.

\

All words in $\mathbb{M}_1^{r}$ are equivalent, and the only word in $\mathbb M_1^{=r}$ is $1^r$.  Furthermore, by Lemma~\ref{lem:ODE}, $\delta_{1^r} = \delta_1^r(a) = 0$ for all $a\in \cA$.  Therefore, taking $K^{1^r} = 0$ satisfies the hypotheses of Theorem~\ref{thm:gutsofproof} for $n = 1$.  Thus, we have that the map $\cN|_{\Gamma_{\mathbb{M}^r_1}}:\Gamma_{\mathbb{M}^r_1}\to \Gamma_{\mathbb{M}^{=r}_1}^{\oplus \cA}$ is surjective. We have then that any locally constant function is an arithmetic partial differential operator. As such, there is $K^\mu \in \Gamma_{\mathbb{M}^r_1} \subset \Gamma$ such that $\cN(K^\mu) = (\delta_\mu(a))_{a\in \cA}$ for each $\mu\in\mathbb{M}^{=r}$. Now, using these $K^\mu$, we may apply Theorem~\ref{thm:gutsofproof} in the full context, to show $\cN:\Gamma\to \Gamma_E^{\oplus \cA}$ is surjective as desired.	
\end{proof}

\end{document}